\documentclass[11pt]{amsart}
\usepackage{amsfonts}
\usepackage[all,cmtip]{xy}
\usepackage{amsmath}
\usepackage{amsthm}
\usepackage{amssymb}
\newtheorem{thm}{Theorem}[section]
\newtheorem{conj}{Conjecture}[section]
\newtheorem{prop}[thm]{Proposition}
\newtheorem*{definition*}         {Definition}
\newtheorem{lemma}[thm]{Lemma}

\newcommand*{\Q}{\mathbb{Q}}
\newcommand*{\Qa}{\overline{\mathbb{Q}}}
\newcommand*{\Z}{\mathbb{Z}}
\newcommand*{\A}{\mathbb{A}}
\newcommand*{\K}{\mathbb{K}}
\newcommand*{\C}{\mathbb{C}}
\newcommand*{\F}{\mathbb{F}}

\newcommand*{\Disc}{\textrm{Disc}}
\DeclareMathOperator{\Gal}{Gal}
\newcommand*{\mult}{\textrm{mult}}
\title{The Existence of an abelian variety over $\Qa$ isogenous to no Jacobian }
\author{Jacob Tsimerman}
\begin{document}
\maketitle
\begin{abstract}
We prove the existence of an abelian variety $A$ of dimension $g$ over $\Qa$ which is not isogenous to any Jacobian, subject to the 
necessary condition $g>3$.
Recently, C.Chai and F.Oort gave such a proof assuming the Andr\'e-Oort conjecture. We modify their proof by constructing a special sequence of CM points for which
we can avoid any unproven hypotheses. We make use of various techniques from the recent work \cite{KY} of Klingler-Yafaev et al.

\end{abstract}

\section{Introduction}

This article is motivated by the following question of Nick Katz and Frans Oort: \emph{Does there exist an abelian variety of genus $g$ over $\Qa$ which is not isogenous 
to a Jacobian of a stable curve?}

For $g\leq 3$ the answer is no because every principally polarized abelian variety is a Jacobian, while for $g\geq 4$ the answer is expected to be yes. 
In \cite{CO}, C.Chai and F.Oort establish this under the Andr\'e-Oort conjecture, which we recall in section 2. In fact, they prove the following stronger statement:

\begin{thm}(\cite{CO})\label{isogenyAO}
Denote by $A_{g,1}$ the coarse moduli space of principally polarized abelian varieties of dimension $g$ defined over $\Qa$, and $X\subsetneq A_{g,1}$ be a proper closed subvariety.
Then assuming the Andr\'e-Oort conjecture, there exists a closed point $y=[A,\lambda]$ in $A_{g,1}$ such that for all points $x=[B,\lambda']$ in $X$, the abelian
varieties $A$ and $B$ are not isogenous.
\end{thm}
The question about Jacobians follows by taking for $X$ the closed Torelli locus. 

	The way Theorem \ref{isogeny} is proven is roughly by looking at the sequence of
all CM points $y$, and using the fact that CM type is preserved under isogeny. Hence, if Theorem \ref{isogeny} is false, $X$ must contain points with
every possible CM type. One then applies the Andr\'e-Oort conjecture to conclude
that $X$ contains a finite set of Shimura sub-varieties containing CM points of each possible CM type. In \cite{CO}, this is ruled out using algebraic methods, 
finishing the proof. 
	
	In \cite{KY}, the Andr\'e-Oort conjecture is proven assuming the Generalized Riemann Hypothesis for Dedekind zeta functions of CM fields, henceforth referred 
to as `GRH'. The reason GRH is used is that they need to produce, for the CM fields $K$ that occur,  
many small split primes\footnote{here small is with respect to the Discriminant $D_K$}. Our idea is to construct an infinite sequence of CM fields which we 
can prove have many small split primes (of course, assuming GRH, they \textbf{all do}). 

We do this in section 3 by using a powerful equidistribution theorem  from Chavdarov\cite{C}, which is due to Nick Katz. We then go into the proof of Andr\'e-Oort 
in \cite{KY}, and carry it through for our sequence of CM points without assuming GRH. Finally, in section 4 we apply the arguments in \cite{CO} to 
our sequence. Thus, our main result is
\begin{thm}\label{isogeny}
Denote by $A_{g,1}$ the coarse moduli space of principally polarized abelian varieties of dimension $g$ over $\Qa$, and $X\subsetneq A_{g,1}$ be a proper closed subvariety.
Then there exists a closed point $y=[A,\lambda]$ in $A_{g,1}$ such that for all points $x=[B,\lambda']$ in $X$, the abelian
varieties $A$ and $B$ are not isogenous.
\end{thm}

We point out that we make no progress on the Andr\'e-Oort conjecture
itself, as the conjecture is about the `worst' possible sequence of CM points, whereas we only show that it holds for certain carefully constructed sequences.

\textbf{Acknowledgments}: It is a pleasure to thank Frans Oort for generously sharing his ideas. Nick Katz was kind enough to clarify
Theorem \ref{averagefrob}. Thanks to Ali Altug who read through a preliminary version of this note and made helpful remarks. Finally, thanks to Peter Sarnak
who introduced the author to the subject matter and helped improve the exposition.

\section{Notation and Background}

\subsection{Weyl CM fields}
Following \cite{CO}, we say that a field $L$ of degree $2g$ is of Weyl CM type if it is a totally complex quadratic extension of a totally real field $F$, and if the Galois group
of the normal closure $M$ of $L$ is $W_g:= (\Z/2\Z)^g\ltimes S_g$, where the action is by permutation of the $\Z/2\Z$'s. 

One can think of this concretely in the 
following
way: let $phi_1,\phi_2,\dots \phi_g$ be $g$ distinct embeddings of $L$ into $\C$, such that no two of them are conjugate. Then 
$\phi_1,\overline{\phi_1},\dots,\phi_g,\overline{\phi_g}$ are all the embeddings of $L$ into $\C$, and hence $\Sigma=(\phi_1,\phi_2,\dots \phi_g)$ is a $CM$-type for $L$. 
Concretely, an element $h\in W_g$ permutes the pairs of embeddings $P_i=(\phi_i,\overline{\phi_i})$. We thus get an element of the group $S_g$ together with $g$ 
choices of sign. Let $S\in W_g$ be the set of elements inducing one of the embeddings $\phi_i$ on $L$. Define $H,H^* \subset W_g$ by 
\begin{align*} 
&h\in H \longleftrightarrow hS=S \\
&h^*\in H^* \longleftrightarrow h^*S^{-1}=S^{-1}.
\end{align*}

One can see that $H$ is the group $W_{g-1}$ of all elements that fix the pair $P_1$ and is the Galois group $\Gal(M/L)$. Also, $H^*$ can be seen to be the set of 
all elements that take each pair $P_i$ to a pair $P_j$ such that $\phi_i$ goes to $\phi_j$. The reflex field of $H^*$ is thus a field $L^*$ of degree $2^g$
with $CM$ type induced by $S^{-1}$. 

\subsection{Shimura Varieties and the Andr\'e-Oort conjecture}

	Here we recall some of the basic theory of Shimura varieties. For more details, we refer to \cite{D1} and\cite{D2}.
A Shimura variety is a pair $(G,X)$, where $G$ is a reductive algebraic group acting on a hermitian symmetric space $X$, 
together with a compact subgroup $K$ of $G(\A_f)$, where $\A_f$ are the finite Adelles. 
Define the space $Sh(G,X)_K := G(\Q)\backslash X\times G(\A_f)/ K$, which is then naturally endowed with
the structure of a quasi-projective algebraic variety over $\Qa$. 
Given another Shimura variety $Sh(G_1,X_1)_{K_1}$ and a pair of morphisms $G_1\rightarrow G, X_1\rightarrow X$ which respect the group actions and
send $K_1$ to $K$, we get a map $Sh(G_1,X_1)_{K_1}\rightarrow Sh(G,X)_K$. A \emph{Shimura subvariety} of $Sh(G,X)_K$ is defined to be
an irreducible component of a ``Hecke translate''
by an element of $G(\A_f)$ of such an image.A Shimura subvariety of 
dimension 0 is called a \emph{special point}.
	
	An important special case of a Shimura variety is the moduli space of principally polarized abelian varieties $A_{g,1}$. 
It corresponds to the pair $(Sp_{2g},\mathbb{H}_g)$ together with the standard maximal compact subgroup of $Sp_{2g}(\A_f)$. 
In this case special points correspond exactly to abelian varieties with complex multiplication.

\begin{conj}(Andr\'e-Oort)\label{AO}
Let $S$ be a Shimura variety, and $\Gamma\subset S$ be a set of special points in $S$. Then the Zariski closure of $\Gamma$ is a finite union of Shimura
subvarieties.
\end{conj}

We call a point $x\in A_{g,1}$ a \emph{Weyl CM point} if the associated abelian variety has complex multiplication by a Weyl CM field of degree $2g$. 

\subsection{Siegel Zeroes and Totally Split Primes}

Later on we shall need to produce totally split primes in algebraic number fields, so we collect the results here for convenience.
Fix $d>0$ throughout this section. Take $K$ to be a Galois extension of $\Q$ of degree $d$ and discriminant $D_K$. For a real number $X>$, define by $N_K(X)$ to 
be the number of primes $p<X$ such that $p$ is a totally split prime in $K$. By Chebotarev's density theorem, we know that $N_K(X)$ is asymptotic to 
$\frac{X}{d\cdot \log(X)}$. However, we shall need a quantified version of this result. For this, we introduce the concept of an exceptional (Siegel) zero:

\begin{thm}
There exists a $C_d>0$ depending only on $d$ such that the Dedekind zeta function $\zeta_K(s)$ has at most one real zero in the range
$$1-\frac{C_d}{\log(D_K)}\leq \sigma < 1.$$ Such a zero, if it exists, is called an exceptional zero, or Siegel zero.
\end{thm}

Exceptional zeroes, though conjectured to not exist, must be entertained all over analytic number theory , and the reason they are important for us is the following result, due to Lagarias and Odlyzko \cite{La}:

\begin{thm}\label{primenumbertheorem}
For $K$ a Galois number field of degree $d$, we have
$$N(K,X) = \frac{X}{\log(X)} + O\left(\frac{X^{\beta}}{\log(X)}\right) + O\left(\frac{\sqrt{|D_K|}Xe^{-C_d\sqrt{\log(X)}}}{\log(X)}\right).$$
where $\beta$ is the possible exceptional zero of $\zeta_K(s)$. 
The $O\left(\frac{X^{\beta}}{\log(X)}\right)$ term should be removed if there is no exceptional zero.
\end{thm}

It is a well established principle that exceptional zeroes, if they exist at all, are very rare. We recall this below and later we shall construct our 
CM fields so as to avoid exceptional zeroes.
By the following result of Heilbronn \cite{H}, exceptional zeroes can genuinely show up only in degree $2$ extensions.

\begin{thm}\label{quadsubfield}
If $K$ is a Galois number field with $\beta$ as an exceptional zero of $\zeta_K(s)$, then there is a quadratic field $F\subset K$ with $\zeta_F(\beta)=0$, so 
that $\beta$ is an exceptional zero of $\zeta_F(s)$.
\end{thm}

For quadratic fields we have the following repulsion result:

\begin{thm}\label{repulsion}
Let $F_1,F_2$ be two distinct quadratic number fields of Discriminants $D_1,D_2$ respectively, and let $\beta_1,\beta_2$ be real zeroes of 
$\zeta_{F_1}(s),$ and $\zeta_{F_2}(s)$ respectively. There exists an absolute constant $c>0$ such that 
$$\min(\beta_1,\beta_2)< 1- \frac{c}{\log(D_1D_2)}.$$
\end{thm}

The proof of the Theorem \ref{repulsion} can be found in Theorem 5.27 of \cite{IK}. Chapter 5 of \cite{IK} is also great introduction to Siegel zeroes
and the analytic theory of L-functions in general.

\section{Producing Weyl CM Fields}

In \cite{KY}, the Andr\'e-Oort conjecture(\ref{AO}) was proven under the assumption of GRH. The reason for their assuming of GRH was to guarantee that certain 
CM fields have many small split primes. As such, our first task is to produce a sequence of Weyl CM fields of fixed degree $g$ containing many small split primes. This is a problem in algebraic number theory. We use methods coming from looking at zeta functions of families of curves over finite fields. It is possible
that one could also accomplish the same task by looking at certain `GRH on average' results, though we have not carried this out. One advantage
of our approach is that we immediately produce CM fields, without having to filter them out. 
In the next section, we follow the methods of \cite{KY} and prove the desired closure property (\ref{AO}) about Zariski Closures for our sequence unconditionally.

We fix an integer $g>1$ and pick a prime number $q>g$, which shall remain fixed for the rest of the section.

In \cite{C}, N.Chavdarov studies the 
following situation : 

 Consider a family of proper, smooth curves
of genus $g$, $\psi: C\rightarrow U$ where $U$ is a smooth affine curve over $\F_q$. Assume that for $l\neq 2,q$ the mod-$l$ monodromy of $R^1\psi_{!}\Z_l$ is the 
full symplectic group $Sp_{2g}(\F_l)$. Such a family can be constructed by taking the family of curves $$\{y^2=(x-t)\prod_{i=1}^{2g}(x-i)\}$$ parametrized
by $t\in \A^1_{\F_q}$, as was proven by Yu (unpublished). The result was also reproven and generalized by Hall in \cite{Ha}.
Fix a symplectic pairing $\langle\, ,\rangle$ and define $$SSp_{2g}(\F_{l}) = \{A\in M_{2g}(\F_l) \mid \langle Av, Aw\rangle = \gamma\langle v,w\rangle\textrm{ for some } \gamma\in\F_l^{\times}\}.$$
We shall use heavily the following Theorem from \cite{C}, where it is attributed to Nick Katz:

\begin{thm}(\cite{C}, Thm 4.1)\label{averagefrob}
With notation as above, let $l_1,l_2,\dots l_r$ be a distinct set of primes not equal to $2$ or $q$. Set 
$G_0=\prod_{i=1}^r Sp_{2g}(\F_{l_i}), G=\prod_{i=1}^r SSp_{2g}(\F_{l_i})$. Then we have the following commutative diagram, where the rows are exact:

$$\xymatrix{
1\ar[r]&\pi_1^{\textrm{geom}}\ar[d]^{\lambda_0}\ar[r]&\pi_1\ar[d]^{\lambda}\ar[r]^{\deg}&\hat{\Z}\ar[d]^{1\rightarrow\gamma^{-1}}\ar[r]&1\\
1\ar[r]&G_0\ar[r]&G\ar[r]^{\mult}&\Gamma\ar[r]&1}$$

For each conjugacy class $C$ of $G$ we have

$$\left|\textrm{Prob}\{u\in U(\F_{q^n})\mid Frob_{u}\in C\} - \frac{|C\cap\mult^{-1}(\gamma^n)|}{|G_0|}\right|\ll_{\psi} |G_0|q^{-n/2}.$$

\end{thm}

In the above theorem the notation $\ll_{\psi} |G_0|q^{-n/2}$ means there exists some constant $c(\psi)>0$ depending only on the family $\psi$ such that
the left hand side is at most $c(\psi) |G_0|q^{-n/2}$. It is critical for us to have the uniform dependence on $G_0$ as the group itself varies.

 For each $u\in U(\F_{q^n})$ we consider the numerator $P_u(T)$ of the zeta function of $C_u$. Theorem 2.3 of \cite{C}
says that $P_u(T)$ is irreducible for a density 1 subset of $U(\overline{\F_{q}})$, where the density of a set $S$ is defined by 

$$\textrm{lim}_{n\rightarrow\infty} \frac{|S\cap U(\F_{q^n})|}{|U(\F_{q^n})|}.$$	

Moreover, the field $\K_u=\Q(\pi_u)$ is a Weyl CM field for a subset of density 1, where $\pi_u$ is a root of $P_u(T)$. We remind the reader that by the Weil 
Conjectures for curves, all conjugates of $\pi_u$ have absolute value $q^{n/2}$. We shall use the fact that how a prime $l\neq q$ factors in $\K_u$ can be read off from 
the image in $SSp_{2g}(\F_l)$ of $Frob_{u}$.  

The idea of the proof is that a conjugacy class mod $l$ tells us how $P_u(T)$ reduces mod $l$. It is proven in \cite{C} that by fixing a finite set of primes
$m_1,m_2,\dots,m_h$ and conjugacy classes $C_i$ in $SSp_{2g}(\F_{m_i})$ one can force $P_u(T)$ to be irreducible and for the associated field to be a Weyl
CM field.

We will now use Theorem \ref{averagefrob} to construct Weyl CM fields $\K_u$ with many small split primes. 
Throughout the rest of this section $n$ will be an integer
parameter that will be tending to infinity, and we shall be picking primes $l_i$ to depend on $n$. First, note that since the ring of integers $O_{\K_u}$ 
contains $\Z[\pi_u]$ as a 
subring of finite index, we have $Disc(\K_u)\leq Disc(\Z[\pi_u])\ll q^{ng^2}$, where the last inequality follows from the fact that all conjugates of $\pi_u$ 
have absolute value $q^{n/2}$. Fix a prime $l$ such that $n^5 < l < 2n^5$. Applying Theorem \ref{averagefrob} to this prime, we see that it splits completely
in $|U(\F_{q^n})|(\frac{1}{2^g g!}+ o_n(1))$ fields $\K_u$. Since this is true for each prime $l$, we see that on average, each field $\K_u$ has 
$$\frac{n^5}{2^g g!\log(n^5)}\cdot(1+o_n(1))$$ primes between $n^5$ and $2n^5$ split completely (Note that since most fields are Weyl CM fields, this is what is expected from 
Chebatorev's density theorem). In particular, there exists at least one $CM$ field $\K_u$ with at least $\frac{n^5}{2^{g+1} g!\log(n^5)}$ primes between $n^5$ 
and $2n^5$ that split completely in $\K_u$. By varying over $n$, we can thus create an infinite such sequence.

We're almost done, but there's still an issue to deal with: We have produced a sequence of Weyl CM fields with lots of split primes, but for these primes to be 
`small' compared to the 
discriminant, we need to ensure that the discriminant of $\K_u$ is large. To accomplish this, recall that 
 a prime $l$ will divide $Disc(\K_u)$ iff $l$ ramifies in $\K_u$, which is to say that $Frob_{u}$ maps to an element of
$SSp_{2g}(\F_l)$ having a repeated root. Pick a finite set of primes $l_1,l_2,\dots,l_r$ such that $l_1l_2l_3\dots l_r$ is on the order of $q^{\frac{n}{32g^2}}.$ 
Using Theorem \ref{averagefrob} for this set of primes and a conjugacy class of $SSp_{2g}(\F_{l_i})$ with repeated roots produces infinitely many CM fields 
$K_n$ which have discriminant $D_n$ divisible by each $l_i$, and therefore satisfying 

$$c_1 q^{\frac{n}{32g^2}}\leq D_n\leq c_2 q^{ng^2}.$$ 

We can now prove the main result of this section:

\begin{lemma}\label{goodweylfields}
For each $g$ there exists a sequence of distinct Weyl CM fields $K_i$ with discriminant $D_i$ satisfying the following properties:
\begin{enumerate}

\item There exists a constant $c_g$ such that at least $c_g\frac{\log(D_i)^5}{\log(\log(D_i))}$ primes $p\leq 2\log(D_i)^5$ split completely in $K_i$.
\item For each number field $L$, the Galois closure of $K_i$ does not contain $L$ for $i\gg_L 0$.
\item There exist $c_1,c_2$ such that $c_1 q^{\frac{n}{32g^2}}\leq D_n\leq c_2 q^{ng^2}.$

\end{enumerate}
\end{lemma}

\begin{proof}

We build the $K_n$ in a few steps. First, we pick a finite set of primes $m_1,m_2,\dots m_h$ and conjugacy classes $C_i$ in the corresponding groups 
We build the $K_n$ in a few steps. First, we pick a finite set of primes $m_1,m_2,\dots m_h$ and conjugacy classes $C_i$ in the corresponding groups 
$SSp_{2g}(\F_{m_i})$ such that $mult(C_i)=\gamma^n$ and for any $u$ with $\lambda_0(Frob_u)\in C_i$, the polynomial $P_u(T)$ is irreducible and 
$\K_u$ is a Weyl CM field. Next, pick for each $n$ primes $l_1,l_2,..\dots l_{r_n}$ distinct from the $m_i$ whose product is asymptotic to 
$q^{\frac{n}{32g^2}}$ 
as $n\rightarrow\infty$ (note that this is easy to do by the prime number theorem). 
Next, we pick conjugacy classes $D_i$ in $SSp_{2g}(\F_{l_i})$ whose characteristic polynomials have repeated roots and such that $mult(D_i)=\gamma^n$.  Finally,
we pick an auxiliary prime $l$ such that $n^5<l<2n^5$ and let $E_l$ denote the union of all conjugacy classes in $SSp_{2g}(\F_l)$ such that $mult(E_l)=\gamma^n$ 
and also the characteristic polynomials of all elements $E_l$ split completely over $\F_l$. We now apply Theorem \ref{averagefrob} to
the primes $m_i,l_j$ with the conjugacy class $C=\prod_{i=1}^h C_i\times\prod_{j=1}^{r_n}D_j$. In the notation of Theorem \ref{averagefrob}, we have
$G=\prod_{i=1}^h SSp_{2g}(\F_{m_i})\times\prod_{j=1}^{r_n}SSp_{2g}(\F_{l_j})$, 
$G_0=\prod_{i=1}^h Sp_{2g}(\F_{m_i})\times\prod_{j=1}^{r_n}Sp_{2g}(\F_{l_j})$ and 

\begin{align*} 
\textrm{Prob}\{u\in U(\F_{q^n})\mid Frob_{u}\in C\}&=\frac{|C\cap G^{\mult\gamma^n}|}{|G_0|} + O(|G_0|q^{-n/2})\\
&=\frac{|C\cap G^{\mult\gamma^n}|}{|G_0|}  + O(q^{-3n/8})\\
\end{align*}

As $|G_0|\ll q^{n/4}$ and $U(\F_{q^n})\asymp q^n$, we see that we have at least $$\frac{U(\F_{q^n})\times |C\cap G^{\mult\gamma^n}|}{|G_0|} + O(q^{5n/8})$$
 points $u$ such that
$\Q(\pi_u)$ is a Weyl CM field $K_u$ with discriminant $$q^\frac{n}{32g^2}\asymp\prod_{j=1}^{r_n} l_j \ll \Disc(K_u)\ll q^{ng^2}$$ so that (3) holds.

Note that different points $u$
could produce the same field $K_u$ so we count the $K_u$ with multiplicity. 
Now, we apply a similar calculation to the primes $m_i,l_j,j$, where now we take the conjugacy class $$C=\prod_{i=1}^h C_i\times\prod_{j=1}^{r_n}D_j\times E_l.$$
shows that of these fields we have the prime $l$ splits 
completely in $$\frac{|E_l|}{|Sp_{2g}(\F_{l})|}\times\frac{U(\F_{q^n})\times |C\cap G^{\mult\gamma^n}|}{|G_0|} + O(|SSP_{2g}(\F_l)| q^{5n/8})$$ of them. 
By (\cite{C}, Theorem 3.5)  it follows
that $\frac{|E_l|}{|Sp_{2g}(\F_{l})|}\longrightarrow \frac{1}{2^g\times g!}$. Averaging over $l$ between $n^5$ and $2n^5$ we see that at
least one of the $K_u$ satisfies condition (1). For condition (2), enumerate all number fields $L_1,L_2,\dots,L_n,\dots$ and pick a totally inert prime $p_i$ in 
each. We can then repeat the above construction of the $K_i$, insisting that $K_n$ is eventually totally split at each of $p_1,p_2,\dots p_m,\dots$ by picking 
appropriate conjugacy classes. This will ensure that (2) holds.

\end{proof}

In order to produce primes later on, we shall need a subsequence of the $K_i$ that has no exceptional zeroes. 

\begin{lemma}\label{greatfields}

There is an infinite subsequence $W_j$ of the $K_i$ such that for $V_j$ the Galois closure of $W_j$, $\zeta_{V_j}(s)$ has no exceptional zero.

\end{lemma}

\begin{proof}

Assume not, so that there is some real number $r$ such that for $D_i\geq r $ the Dedekind zeta function 
$\zeta_{L_i}(s)$ has an exceptional zero, where $L_i$ is the Galois closure of $K_i$. By Theorem \ref{quadsubfield} this implies that there is a quadratic
subfield $F_i\subset L_i$ such that $\zeta_{F_i}(s)$ has a zero $\beta_i$ such that $$1-\frac{C_{2^g\cdot g!}}{\log{D_i}}<\beta_i<1.$$ By (3) of 
lemma \ref{goodweylfields} there is some $K_j$ with $D_j>D_i>r$ such that $$1-\frac{32g^4\cdot C_{2^g\cdot g!}}{\log{D_j}}<\beta_i<1-\frac{C_{2^g\cdot g!}}{\log{D_j}}.$$

Hence there is some quadratic field $F_j\subset L_j$ such that $\zeta_{F_j}(s)$ has a zero $\beta_j$ with $$1-\frac{C_{2^g\cdot g!}}{\log{D_j}}<\beta_j<1.$$
However, note that  $$\log\left(|\Disc(F_i)|\cdot |\Disc(F_j)|\right)\leq \log(|\Disc(L_j)|)\leq2^g\cdot g!\log(D_j).$$ Applying Theorem \ref{repulsion}
we arrive at $$1-\frac{32g^4 C_{2^g\cdot g!}}{\log{D_j}}<\min(\beta_1,\beta_2)<1-\frac{c}{\log\left(|\Disc(F_i)|\cdot |\Disc(F_j)|\right)}<
1-\frac{c}{2^g\cdot g!\log(D_j)}.$$

By taking $C_{2^g\cdot g!}< c $ we arrive at a contradiction, as desired.

\end{proof}

\section{Proof of Theorem \ref{isogeny}}

In this section we combine the arguments of \cite{CO} with our lemma \ref{goodweylfields} to prove Theorem \ref{isogeny}. First, we recall the following
bound of Yafaev:

\begin{lemma}(Yafaev)

Fix a Shimura variety $Sh(G,X)_K$ defined over a number field $F$. For any $\epsilon>0$ and $N>0$, there exist $c_1,c_2>0$ such that the following holds:
	
Let $s$ be a special point, with CM by a field $K$, in $Sh(G,X)$ . Let $K$ have discriminant $D_k$, and suppose there are at least 
$\epsilon\frac{\log(D_K)}{\log(\log(D_K))}$ primes $p < \frac{1}{\epsilon}(\log D_K)^5$  that split completely in $K$. Then

$$|\textrm{Gal}(\Qa/F)\cdot s|\geq c_1 \times\log D_K)^N \cdot \prod_{\substack{\textrm{p prime}\\ \textrm{MT(s)}_{/\F_p} \textrm{is not a torus}}}c_2p$$

Where $MT(s)$ denotes the Mumford-Tate group associated to $s$.
\end{lemma}

\begin{proof}
The above is Theorem 2.1 in \cite{Y}. The theorem is stated with the assumptions of GRH, but this assumption is only used in Theorem 2.15 to produce small 
split primes, whose existence we are assuming in the statement of the lemma.
\end{proof}

Before proceeding with the proof of Theorem \ref{isogeny}, we make a definition: Following \cite{CO}, we define a Hilbert modular variety attached to a totally
real field $F$ of degree $g$ over $Q$ to be any irreducible component of a closed subvariety $A^{\mathcal{O}}_{g,1}\subset A_{g,1}$ over $\Qa$. 
Here $\mathcal{O}$ is an order in $F$ and $A^{\mathcal{O}}_{g,1}$ is the locus of all points $[A,\lambda]$ where the endomorphisms ring of $A$ contains 
$\mathcal{O}$ as a subring. Note that each Hilbert modular variety is a Shimura subvariety of $A_{g,1}$ corresponding to the pair 
$(\textrm{Res}_{F/\Q}SL_2,\mathbb{H}^g)$.

\begin{lemma}\label{hilbertmodular}
If $S\subsetneq A_{g,1}$ is a positive dimensional Shimura subvariety which contains a Weyl CM point, then $S$ is a Hilbert modular subvariety.
\end{lemma}
\begin{proof}
This is lemma 3.5 in \cite{CO}.
\end{proof}

\begin{proof} \emph{of Theorem \ref{isogeny}:}

Pick a sequence of principally polarized abelian varieties $y_i$ such that $y_i$ has complex multiplication
by the field $W_i$, where $W_i$ are the Weyl CM fields constructed in Lemma \ref{greatfields}. That one can do this is a standard fact in the theory of 
abelian varieties, see \cite{S} for details. Assume the statement of the theorem is false. Then $X$ contains $x_i$ such that $x_i$ is isogenous to
$y_i$ and therefore has complex multiplication by $W_i$. If Theorem 8.3.1 in \cite{KY} holds for $Z=X$ and $V=x_i$, then for $i\gg0$ we can conclude that
 $X$ contains a Shimura subvariety $S_i$ containing $x_i$. By Lemma \ref{hilbertmodular}, $S_i$ must be a Hilbert modular variety. Moreover, the $S_i$
 form an infinite set since the $W_i$ eventually have distinct totally real subfields by (2) of Lemma \ref{goodweylfields}. However, by
 Theorem 1.2 of \cite{CU}, some subsequence $S_{n_i}$ becomes equidistributed for the unique homogeneous measure corresponding to a Shimura subvariety $S\subset A_{g,1}$ 
which must contain $S_{n_i}$ for large enough $i$. We can thus conclude that $S$ is not a finite union of Hilbert modular varieties, and so 
by lemma \ref{hilbertmodular} this means that $S$ must be all of $A_{g,1}$. The $S_i$ thus become 
equidistributed for the natural measure in $A_{g,1}$, which is a contradiction to $S_i\subset X$. Hence, its enough to 
verify Theorem 8.3.1 of \cite{KY} in our case.
 
 Now, the assumption of GRH in Theorem 8.3.1 is used only in Proposition 9.1 of \cite{KY} to produce a small prime $l$ as in the following
proposition \ref{smallprime}. By proving the following proposition unconditionally, we complete the proof of Theorem \ref{isogeny}.

\begin{definition*}
Fix a positive constant $B>0$. Define $\beta_i$ to be $\beta_i = \prod_{p} (Bp)$ where the product goes over all primes $p$ such that $MT(x_i)_{/\F_p}$ is not a torus.
\end{definition*} 

From now on $D_i$ will denote the discriminant of $W_i$.
\begin{prop}\label{smallprime}

Fix $\epsilon>0, c>0$. Then for each $i\gg0$ there exists a prime $l$ such that:
\begin{enumerate}

\item $l$ is totally split in $W_i$.
\item $MT(x_i)_{/\F_l}$ is a torus
\item $l < c \log(D_i)^{6}\beta_i^{\epsilon}$.

\end{enumerate}

\end{prop}

\begin{proof}

By construction, there is a constant $c_g$ such that there are at least $$c_g\frac{(\log D_i)^5}{\log(\log(D_i))}$$ primes $p\leq 2\log(D_i)^5$ split completely in
$W_i$. Since $\beta_i$ is bounded from below (there are only finitely many primes less than $B$) we see that for $i\gg0$ all these primes satisfy conditions
(1) and (3). We are thus done unless $MT(x_i)_{/\F_p}$ is not a torus for all these primes $p$. Assume this is the case from now on. We thus have 
\begin{equation}\label{useful}
\beta_i\gg e^{(\log D_i)^4}.
\end{equation}

By Theorem \ref{primenumbertheorem}, for $X\gg e^{(\log D_i)^3}$, the number of totally split primes in $W_i$ less than $X$ is 
$$\pi_{W_i}(X)=\frac{1}{2^g\cdot g!}\cdot\frac{X}{\log(X)} + o(\frac{X}{\log(X)})$$ since by construction the Dedekind zeta function $\zeta_{V_i}(s)$ has no 
exceptional zero, where we define $V_i$ do be the Galois closure of $W_i$. Thus, for $i\gg0$ we have $$\pi_{W_i}(X)\gg \frac{X}{\log(X)}.$$
Since for large enough $i$ we have $e^{(\log D_i)^3}< c \log(D_i)^{6}\beta_i^{\epsilon}$, there are at least 
$\frac{\beta_i^{\epsilon}}{\epsilon\log(\beta_i)}$ totally split primes $l$ in $W_i$ such that $l<c \log(D_i)^{6}\beta_i^{\epsilon}$ for large enough $i$. Now, 
one of these primes $l$ must be such that $MT(x_i)_{/\F_l}$ is a torus, since otherwise we would have $$\beta_i\gg 
2^{\frac{\beta_i^{\epsilon}}{\epsilon\log(\beta_i)}},$$ which is 
false for large enough $i$ by equation (\ref{useful}). This completes the proof.

\end{proof}

\end{proof}

\end{document}